\documentclass[a4,12pt]{amsart}
\usepackage{amssymb,amstext,amsmath,amscd,amsthm,amsfonts,enumerate,latexsym}
\usepackage{color}
\usepackage[dvipdfmx]{graphicx}
\usepackage[all]{xy}
\usepackage{stmaryrd} 
\usepackage{ytableau}
\usepackage[dvipdfmx]{hyperref}

\usepackage{bm}

\usepackage{mathptmx}

\theoremstyle{plain}
\newtheorem{thm}{Theorem}[section]

\newtheorem*{thm*}{Theorem}
\newtheorem*{cor*}{Corollary}

\newtheorem{prop}[thm]{Proposition}

\newtheorem{lem}[thm]{Lemma}
\newtheorem{cor}[thm]{Corollary}

\newtheorem*{claim*}{Claim}

\theoremstyle{definition}

\newtheorem{ex}[thm]{Example}

\newtheorem{case}{Case}

\newtheorem{setup}[thm]{Setup}

\theoremstyle{remark}
\newtheorem{rem}[thm]{Remark}

\numberwithin{equation}{thm}

\newtheorem*{ac}{Acknowledgments}

\def\Ext{\operatorname{Ext}}

\def\Ker{\operatorname{Ker}}

\def\Hom{\operatorname{Hom}}

\def\Assh{\operatorname{Assh}}

\def\Proj{\operatorname{Proj}}

\def\m{\mathfrak m}

\def\p{\mathfrak p}
\def\q{\mathfrak q}

\newcommand{\rme}{\mathrm{e}}

\newcommand{\rmE}{\mathrm{E}}

\newcommand{\rmH}{\mathrm{H}}

\newcommand{\rmK}{\mathrm{K}}

\newcommand{\calF}{\mathcal{F}}

\newcommand{\calR}{\mathcal{R}}

\newcommand{\calX}{\mathcal{X}}

\newcommand{\fkm}{\mathfrak{m}}

\newcommand{\fkq}{\mathfrak{q}}

\newcommand{\fkM}{\mathfrak{M}}
\newcommand{\fkN}{\mathfrak{N}}

\newcommand{\mapright}[1]{%
\smash{\mathop{%
\hbox to 1cm{\rightarrowfill}}\limits^{#1}}}

\newcommand{\mapleft}[1]{%
\smash{\mathop{%
\hbox to 1cm{\leftarrowfill}}\limits_{#1}}}

\def\depth{\operatorname{depth}}

\def\Supp{\operatorname{Supp}}

\def\Ass{\operatorname{Ass}}

\def\Spec{\operatorname{Spec}}

\title[Quasi-Gorenstein extended Rees algebras associated with filtrations]{Quasi-Gorenstein extended Rees algebras \\ associated with filtrations}

\author[Naoki Endo]{Naoki Endo}
\address{School of Political Science and Economics, Meiji University, 1-9-1 Eifuku, Suginami-ku, Tokyo 168-8555, Japan}
\email{endo@meiji.ac.jp}
\urladdr{https://www.isc.meiji.ac.jp/~endo/}

\thanks{2020 {\em Mathematics Subject Classification.} 13A30, 13H10, 13C14.}
\thanks{{\em Key words and phrases.} Quasi-Gorenstein ring, extended Rees algebra, FLC ring}
\thanks{The author was partially supported by JSPS Grant-in-Aid for Scientific Research (C) 23K03058.} 

\begin{document}

\maketitle

\setlength{\baselineskip}{15pt}

\begin{abstract}
This paper investigates the quasi-Gorenstein property of extended Rees algebras associated with the Hilbert filtrations on a Noetherian local ring. We provide necessary and sufficient conditions for the deformation of the quasi-Gorenstein property, characterized by the Cohen-Macaulayness of the Matlis dual of local cohomology modules. As a consequence, we offer a characterization of the quasi-Gorenstein property of extended Rees algebras in terms of conditions on the length of local cohomology.
\end{abstract}


\section{Introduction}\label{sec1}

In this paper, we study the quasi-Gorenstein property for extended Rees algebras associated with the Hilbert filtrations of ideals in a Noetherian local ring. 
As introduced by Platte and Storch in 1977 \cite{PS}, a Noetherian local ring $(R, \m)$
 is defined to be {\it quasi-Gorenstein} if it possesses a canonical module $\rmK_R$ such that $R \cong \rmK_R$ as an $R$-module. Setting $d = \dim R$, this condition is equivalent to $\rmH_\m^d(R) \cong \rmE_R(R/\m)$, where $\rmH_\m^d(-)$ denotes the $d$-th local cohomology functor with respect to $\m$, and $\rmE_R(R/\m)$ represents the injective envelope of $R/\m$. 
Now, let $R = \bigoplus_{n \in \mathbb{Z}} R_n$ be a Noetherian $\Bbb Z$-graded ring with unique graded maximal ideal $\fkM$. We define $R$ to be {\it quasi-Gorenstein} if it has a graded canonical module and its localization $R_{\fkM}$ is quasi-Gorenstein. Equivalently, $R$ admits a graded canonical module $\rmK_R$ such that $\rmK_R \cong R(a)$ for some $a \in \Bbb Z$. 
Here, for a graded $R$-module $M$ and for an integer $\ell$, let $M(\ell)$ denote the graded $R$-module whose underlying $R$-module is the same as that of the $R$-module $M$ and the grading is given by $[M(\ell)]_m = M_{\ell + m}$ for all $m \in \Bbb Z$, where $[-]_m$ denotes the $m$-th homogeneous component.

The quasi-Gorenstein property of the extended Rees algebra has been studied in previous works, such as \cite{HKU, HKU2, Kim}. This paper focuses on investigating the conditions under which quasi-Gorensteinness deforms, specifically examining its inheritance from the associated graded ring to the extended Rees algebra. 
The question of whether the quasi-Gorenstein property deforms -- that is, whether $R$ is quasi-Gorenstein if a Noetherian local ring $(R, \m)$ and its non-zerodivisor $x \in \m$ satisfy $R/xR$ being quasi-Gorenstein -- has been a fundamental question in commutative ring theory. In 2020, Shimomoto, Taniguchi, and Tavanfar provided a counterexample, using Macaulay2, demonstrating that quasi-Gorensteinness does not deform in general (\cite[Theorem 4.2]{STT}). Thus, characterizing the circumstances under which quasi-Gorensteinness deforms has become a central problem in this area. 
In \cite{STT}, the authors presented various sufficient conditions under which the deformation of quasi-Gorensteinness holds. Moreover, they discussed conditions for the deformation of quasi-Gorensteinness in graded rings, but their results were restricted to $\Bbb N$-graded rings, particularly standard graded rings. In contrast, this paper examines the quasi-Gorenstein property of extended Rees algebras, which are $\Bbb Z$-graded rings associated with filtrations of ideals, in relation to the quasi-Gorensteinness of their associated graded rings. In particular, this study characterizes the deformation of quasi-Gorensteinness from the perspective of the Cohen-Macaulayness of the Matlis dual of local cohomology modules.

Let $(R, \m)$ be a Noetherian local ring with $d=\dim R \ge 3$ which is a homomorphic image of a Gorenstein ring. Let $\calF=\{F_n\}_{n \in \Bbb Z}$ be the Hilbert filtration of ideals in $R$, i.e., it is a filtration of ideals such that $F_1$ is $\m$-primary and $F_{n+1} = F_1F_n$ for all $n \gg 0$. Denote by
$$
\calR'(\calF) = \sum_{n \in \Bbb Z}F_nt^n \subseteq R[t, t^{-1}] \ \ \text{and} \ \ G(\calF) = \calR'(\calF)/t^{-1}\calR'(\calF) \cong \bigoplus_{n \ge 0}F_n/F_{n+1}
$$
the {\it extended Rees algebra of $\calF$} and the {\it associated graded ring of $\calF$}, respectively, where $t$ is an indeterminate over $R$. 

With this notation, the main result of this paper is stated as follows.

\begin{thm}\label{main}
Suppose that $G(\calF)$ is a quasi-Gorenstein graded ring, $\depth \calR'(\calF) \ge d$, and $\rmH^{d-1}_{\fkM}(G(\calF))$ is finitely generated as an $\calR'(\calF)$-module, where $\rmH_\fkM^i(-)$ denotes the $i$-th graded local cohomology functor with respect to the unique graded maximal ideal $\fkM$ of $\calR'(\calF)$.
Then $\calR'(\calF)$ is quasi-Gorenstein if and only if the length of $\rmH^{d-1}_{\m}(R)$ as an $R$-module coincides with the length of $\rmH^{d-1}_{\fkM}(G(\calF))$ as a $G(\calF)$-module.
\end{thm}

This paper is organized as follows. In Section 2, we review fundamental concepts and results pertaining to canonical modules, modules with finite local cohomology (FLC), and blow-up algebras associated with filtrations of ideals. Additionally, we present a refined version of the well-known result \cite[Corollary 2.7]{VV} concerning regular sequences on associated graded rings. Section 3 focuses on the conditions under which the quasi-Gorenstein property is preserved under deformations, utilizing the Cohen-Macaulay property of the Matlis dual of local cohomology modules. Finally, in Section 4, we provide a proof of Theorem \ref{main}.

Throughout this paper, unless otherwise specified, we use the following terminology and notation. For a commutative ring $R$ and an $R$-module $N$, let $\ell_R(N)$ denote the length of $N$. When $(R, \m)$ is a Noetherian local ring, we denote by $\widehat{R}$ the $\m$-adic completion of $R$. The Matlis dual functor is denoted by $(-)^{\vee} = \Hom_R(-, \rmE_R(R/\m))$, where $\rmE_R(R/\m)$ is the injective envelope of $R/\m$. Let $\rmH_\m^i(-)$ be the $i$-th local cohomology functor with respect to $\m$. Furthermore, for an $\m$-primary ideal $I$ in $R$ and a finitely generated $R$-module $M$ with $s = \dim_RM$, there exist integers $\rme_i(I, M)$, called the {\it Hilbert coefficients} of $M$ with respect to $I$, satisfying the equality
$$
\ell_R(M/I^{n+1}M) = \rme_0(I, M) \binom{n+s}{s} - \rme_1(I, M) \binom{n+s-1}{s-1} + \cdots + (-1)^s \rme_s(I, M)
$$
for all $n \gg 0$. When $R = \bigoplus_{n \in \Bbb Z}R_n$ is a Noetherian $\Bbb Z$-graded ring with unique graded maximal ideal $\fkM$, we denote by $\rmH_\fkM^i(-)$ the $i$-th graded local cohomology functor with respect to $\fkM$.


\section{Preliminaries}

In this section, we provide an overview of the preliminaries that will be utilized throughout this paper. Let $(R, \m)$ be a Noetherian local ring with $d = \dim R$. 
For a finitely generated $R$-module $M$, we define $\Assh_RM = \{\p \in \Supp_RM \mid \dim R/\p = \dim_RM\}$. For an ideal $I$ of $R$, let $V(I)$ denote the set of all prime ideals of $R$ containing $I$.

\subsection{Canonical modules}
Recall that a {\it canonical module} $K$ of $R$ is a finitely generated $R$-module satisfying the isomorphism
$$
\widehat{R} \otimes_{R} K \cong \Hom_{\widehat{R}}(\rmH^d_{\widehat{\m}}(\widehat{R}), \rmE_{\widehat{R}}(\widehat{R}/\widehat{\m})) 
$$
(\cite[Definition~5.6]{HK}). The canonical module is uniquely determined up to isomorphism (\cite[(1.5)]{A}; see also \cite[Lemma~5.8]{HK}), provided it exists. We denote the canonical module by $\rmK_R$. A canonical module exists for the ring $R$ if $R$ is a homomorphic image of a Gorenstein ring, and the converse also holds when $R$ is Cohen-Macaulay (\cite{Re, S2}).

Now assume that the canonical module $\rmK_R$ exists. For every $\p \in \Supp_R \rmK_R$, the localization $(\rmK_R)_\p$ serves as the canonical module of $R_\p$ (\cite[(1.6)]{A}). Moreover, it holds that 
$$
\Supp_R \rmK_R = \{\p \in \Spec R \mid \dim R_\p + \dim R/\p = d \} 
$$
(\cite[(1.9)]{A}). 
Additionally, for every $\p \in \Supp_R \rmK_R$, any subsystem of parameters for $R_\p$ of length at most $2$ forms a $\rmK_R$-regular sequence. Thus, $\rmK_R$ satisfies Serre's $(S_2)$-condition (\cite[(1.10)]{A}). Here, a finitely generated $R$-module $M$ satisfies {\it Serre's $(S_n)$-condition}, if $\depth_{R_\p}M_\p \ge \inf\{n, \dim_{R_\p}M_\p\}$ for every $\p \in \Spec R$.


\subsection{FLC modules}

Following \cite{STC}, a finitely generated $R$-module $M$ is said to have \textit{finite local cohomology} (abbr. FLC) if $\rmH^i_{\m}(M)$ is finitely generated (or equivalently, of finite length) for all $i \neq \dim_R M$. For instance, if either $\dim_R M \leq 1$ or $M$ is a Buchsbaum module (e.g., Cohen-Macaulay), then $M$ possesses FLC. Thus, modules with FLC can be regarded as a generalization of Buchsbaum modules, and they are occasionally referred to as \textit{generalized Buchsbaum} modules or \textit{generalized Cohen-Macaulay} modules. 


Note that $M$ has FLC if and only if the $\widehat{R}$-module $\widehat{M}$ does as well. When $s = \dim_RM \ge 1$, the condition for the $R$-module $M$ to have FLC is equivalent to the condition that the supremum ${\Bbb I}(M) = \sup_\q (\ell_R(M/\q M) - \rme_0(\q, M))$, taken over all parameter ideals $\q$ of $M$, is finite (\cite[(3.3) Satz]{STC}, \cite[Lemma 1.1]{T}). Furthermore, it is also equivalent to the condition that there exists an integer $\ell \gg 0$ such that every system of parameters contained in $\m^{\ell}$ acts as a $d$-sequence on $M$ (\cite[Theorem]{GO}). In this case, the equality 
$$
{\Bbb I}(M) = \sum_{i=0}^{s-1}\binom{s-1}{i}\ell_R(\rmH^i_\m(M))
$$ holds (\cite[(3.7) Satz]{STC}, \cite[Lemma 1.5]{T}). 
If $M$ has FLC with $s=\dim_RM \ge 1$, the localization $M_\p$ at $\p \in \Supp_RM \setminus\{\m\}$ is necessarily a Cohen-Macaulay $R_\p$-module satisfying the equality 
$$
\dim_RM=\dim_{R_\p}M_\p + \dim R/\p.
$$ 
The converse holds if $R$ is a homomorphic image of a Cohen-Macaulay ring (\cite[(2.5) Satz]{STC}, \cite[(1.17)]{AG}, see also \cite[Corollary 1.2]{K}). 
Consequently, all normal isolated singularities that appear as homomorphic images of Cohen-Macaulay rings possess FLC. 
For more details on FLC modules, the reader may refer to \cite{AG, GO, HIO, STC, SV, T}, and others.
 


\subsection{Blow-up algebras associated with filtrations}

Let $\calF=\{F_n\}_{n \in \Bbb Z}$ be a filtration of ideals in $R$, meaning that $F_n$ is an ideal of $R$, $F_n \supseteq F_{n+1}$, $F_m F_n \subseteq F_{m+n}$ for all $m, n \in \Bbb Z$, and $F_0 = R$.  
Define
$$
\calR'(\calF) = \sum_{n \in \Bbb Z}F_nt^n \subseteq R[t, t^{-1}] \ \ \text{and} \ \ G(\calF) = \calR'(\calF)/t^{-1}\calR'(\calF) \cong \bigoplus_{n \ge 0}F_n/F_{n+1},
$$
as the {\it extended Rees algebra of $\calF$} and the {\it associated graded ring of $\calF$}, respectively, where $t$ is an indeterminate over $R$.
Recall that the filtration $\calF=\{F_n\}_{n \in \Bbb Z}$ is {\it Noetherian}, if the ring $\calR'(\calF)$ is Noetherian. We say that the filtration $\calF=\{F_n\}_{n \in \Bbb Z}$ is 
{\it Hilbert}, if $F_1$ is $\m$-primary and $F_{n+1} = F_1F_n$ for all $n \gg 0$.
In addition, for an ideal $I$ of $R$, we denote by $\calF/I$ the filtration $\{(F_n + I)/I\}_{n \in \Bbb Z}$ of ideals in $R/I$. 


Examples of Hilbert filtrations are plentiful. Beyond the classical instance of the ideal-adic filtration, when $R$ is analytically unramified, the filtration $\{\overline{I^n}\}_{n \in \Bbb Z}$, consisting of the integral closures of powers of an $\m$-primary ideal $I$, constitutes the Hilbert filtration (\cite{Rees}). Furthermore, the filtration $\{\widetilde{(I^n)}\}_{n \in \Bbb Z}$, formed by the Ratliff-Rush closures of the powers of $I$, also qualifies as the Hilbert filtration.

The following constitutes a generalization of the works of \cite[Proposition 6]{H} and \cite[Corollary 2.7]{VV}, while also serving as a partial generalization of 
\cite[Proposition 3.5]{HM}. Although this may be familiar to experts in the field, we include a proof for the sake of completeness.

\begin{prop}\label{regseq}
Let $(R, \m)$ be a Noetherian local ring and $\calF =\{F_n\}_{n \in \Bbb Z}$ the Hilbert filtration of ideals in $R$. Let $a_1, a_2, \ldots, a_r \in R~(r>0)$ such that $a_i \in F_{n_i}$ with $n_i \ge 0$. Then the following conditions are equivalent.
\begin{enumerate}
\item[$(1)$] $a_1t^{n_1}, a_2t^{n_2}, \ldots, a_rt^{n_r} \in \calR'(\calF)$ forms a regular sequence on $G(\calF)$.
\item[$(2)$] $a_1, a_2, \ldots, a_r$ forms a regular sequence on $R$ and the equality $(a_1, a_2, \ldots, a_r) \cap F_n = \displaystyle\sum_{j=1}^r a_jF_{n-n_j}$ holds for all $n \in \Bbb Z$.
\end{enumerate}
When this is the case, one has an isomorphism
$$
G(\calF)/(a_1t^{n_1}, a_2t^{n_2}, \ldots, a_rt^{n_r})G(\calF) \cong G(\calF/(a_1, a_2, \ldots, a_r))
$$
of rings.
\end{prop}

\begin{proof}
$(1) \Rightarrow (2)$ By induction on $r$. Suppose $r=1$. Let $x \in R$ with $a_1 x = 0$. Since the Hilbert filtration is separated, if $x \ne 0$, we can choose $n \ge 0$ such that $x \in F_n \setminus F_{n+1}$. Then
$$
a_1 t^{n_1} \cdot \overline{xt^n} = \overline{(a_1x)t^{n_1+n}} = 0
$$
where $\overline{(-)}$ denotes the image in $G(\calF)$. Since $a_1 t^{n_1}$ is $G(\calF)$-regular, it follows that $\overline{xt^n}=0$ in $G(\calF)$. This implies $x \in F_{n+1}$, which contradicts $x \in F_n \setminus F_{n+1}$. Hence $x=0$. If $a_1 \not\in \m$, then $n_1=0$ because $a \in F_{n_1}$. Thus, $a_1 t^{n_1} = a_1$ is a unit in $\calR'(\calF)$, a contradiction. Therefore $a \in \m$ is $R$-regular. Next, we will show that $(a_1) \cap F_n = a_1 F_{n-n_1}$ holds for all $n \in \Bbb Z$. Indeed, let $x \in (a_1) \cap F_n$, and write $x = a_1 y$ for some $y \in R$. Suppose $y \notin F_{n-n_1}$. Choose $\ell \ge 0$ such that $y \in F_{\ell}\setminus F_{\ell+1}$. Then $n-n_1 \ge \ell + 1$. Consequently, $a_1 y = x \in F_n \subseteq F_{n_1 + \ell + 1}$. Since $a_1 t^{n_1}$ is $G(\calF)$-regular, we have $y \in F_{\ell+1}$, a contradiction. Thus, $y \in F_{n-n_1}$, and hence $(a_1) \cap F_n \subseteq a_1 F_{n-n_1}$. The converse inclusion follows from $a_1 \in F_{n_1}$. Therefore, $(a_1) \cap F_n = a_1 F_{n-n_1}$. 
The canonical surjection $\varepsilon : R \to R/(a_1)$ induces a surjective graded ring homomorphism 
$$
\varphi : G(\calF) \to G(\calF/(a_1)) 
$$
defined by $\varphi(\overline{xt^n}) = \overline{(x+F_{n+1})t^n}$ for each $x \in F_n$. For $x \in F_n$, we have $xt^{n} \in \Ker \varphi$ if and only if $x \in (F_{n+1} + (a_1)) \cap F_n$, or equivalently $x \in F_{n+1} + a_1F_{n-n_1}$, because 
$(a_1) \cap F_n=a_1F_{n-n_1}$. Hence $\Ker \varphi = (a_1t^{n_1})$, and therefore $
G(\calF)/(a_1t^{n_1})G(\calF) \cong G(\calF/(a_1))$. Thus, $(2)$ holds for $r=1$.

Now, assume $r \ge 2$ and that $(2)$ holds for $r-1$. Since $G(\calF)/(a_1t^{n_1})G(\calF) \cong G(\calF/(a_1))$, the sequence $a_2t^{n_2}, a_3t^{n_3}, \ldots, a_rt^{n_r}$ is $G(\calF/(a_1))$-regular. Note that $\calF/(a_1)$ is the Hilbert filtration of ideals in $R/(a_1)$. By the induction hypothesis, $a_2, a_3, \ldots, a_r$ is an $R$-regular sequence, and the equality
$$
(a_2, a_3, \ldots, a_r) \overline{R} \cap F_n \overline{R} = \sum_{j=2}^r a_jF_{n-n_j}\overline{R}
$$
holds, where $\overline{R} = R/(a_1)$. In particular, $a_1, a_2, \ldots, a_r \in \m$ forms an $R$-regular sequence. Finally, let $x \in (a_1, a_2, \ldots, a_r) \cap F_n$. We can choose $y_i \in F_{n-n_i}$ such that $x-(a_2y_2 + a_3y_3 + \cdots + a_ry_r) \in (a_1)$. Since $a_iy_i \in F_n$ for all $2 \le i \le r$, it follows that $x-(a_2y_2 + a_3y_3 + \cdots + a_ry_r) \in (a_1) \cap F_n = a_1 F_{n-n_1}$. Thus $x \in  \sum_{j=1}^r a_jF_{n-n_j}$. Consequently, $(a_1, a_2, \ldots, a_r) \cap F_n =\sum_{j=1}^r a_jF_{n-n_j}$. Furthermore, the induction hypothesis implies
\begin{eqnarray*}
G(\calF)/(a_1t^{n_1}, a_2t^{n_2}, \ldots, a_rt^{n_r})G(\calF) \!\!\!\!&\cong&\!\!\!\! G(\calF/(a_1))/(a_2t^{n_2}, a_3t^{n_3}, \ldots, a_rt^{n_r})G(\calF/(a_1)) \\
\!\!\!\!&\cong&\!\!\!\! G(\left(\calF/(a_1)\right)/(a_2, a_3, \ldots, a_r)\left(\calF/(a_1)\right)) \\
\!\!\!\!&\cong&\!\!\!\! G(\calF/(a_1, a_2, \ldots, a_r)). 
\end{eqnarray*}
Thus, the result follows.

$(2) \Rightarrow (1)$ Let $J_i = (a_1, a_2, \ldots, a_i)$ for each $1 \le i \le r$. We prove by descending induction that
$$
J_i \cap F_n = \sum_{j=1}^ra_j F_{n-n_j}
$$
holds for all $1 \le i \le r$ and $n \in \Bbb Z$. 
Assuming the assertion holds for $i+1$, we proceed by induction on $i$. Without loss of generality, we assume $1 \le i < r$ and that the assertion holds for $i+1$. We now further proceed by induction on $n$. Notice that the desired equality holds trivially for $n \le 0$. Let $n > 0$, and assume that the equality holds for $n-1$. By defining $L= \sum_{j=1}^{i}a_j F_{n-n_j}$, we obtain
$$
J_i \cap F_n \subseteq J_{i+1} \cap F_{n} = \sum_{j=1}^{i+1}a_j F_{n-n_j} = L + a_{i+1}F_{n-n_{i+1}}.
$$
Thus, $J_i \cap F_n \subseteq \left[L + (a_{i+1}F_{n-n_{i+1}})\right] \cap J_i = L + \left[J_i \cap (a_{i+1}F_{n-n_{i+1}})\right]$. It remains to verify that
$$
J_i \cap (a_{i+1}F_{n-n_{i+1}}) \subseteq L. 
$$
Let $x \in J_i \cap (a_{i+1}F_{n-n_{i+1}})$, and write $x = a_{i+1}y$ with $y \in F_{n-n_{i+1}}$. As $a_{i+1}y =x \in F_i$ and $a_{i+1}$ is a non-zerodivisor modulo $J_i$, we have $y \in J_i \cap F_{n-n_{i+1}}$. 
\begin{case}
$n_{i+1}>0$
\end{case}
In this case,  $n-n_{i+1} < n$. By the induction hypothesis on $n$, we have 
$
y \in J_i \cap F_{n-n_{i+1}} = \sum_{j=1}^ia_j F_{(n-n_{i+1})-n_j}.
$
Since $a_{i+1} \in F_{n_{i+1}}$, it follows that $x = a_{i+1}y \in \sum_{j=1}^i a_j F_{(n-n_{i+1})-n_j}\cdot F_{n_{i+1}} \subseteq \sum_{j=1}^{i}a_j F_{n-n_j} = L$,
as desired. 
\begin{case}
$n_{i+1} = 0$
\end{case}
We first show that $J_i \cap (a_{i+1}F_n) \subseteq (a_{i+1}^m F_n) + L$ for all $m>0$. If $m = 1$, this inclusion is immediate. Assume $m > 1$, and that the inclusion holds for  $m-1$. Let $x \in J_i \cap (a_{i+1}F_n)$. Then $x = a_{i+1}^{m-1}y + \ell$ for some $y \in F_n$ and $\ell \in L$. Thus, $a_{i+1}^{m-1}y = x - \ell \in J_i$, and since $a_{i+1}$ is $R/J_i$-regular, it follows that $y \in J_i$. Therefore $y \in J_i \cap F_n \subseteq J_{i+1} \cap F_n \subseteq L+ a_{i+1}F_n$. Hence $x =a_{i+1}^{m-1}y + \ell \in a_{i+1}^m F_n + L$, that is, $J_i \cap (a_{i+1}F_n) \subseteq (a_{i+1}^m F_n) + L$. 
Taking the intersection over all $m>0$, we obtain 
$$
J_i \cap (a_{i+1}F_n) \subseteq \bigcap_{m>0} \left[(a_{i+1}^m F_n) + L\right] \subseteq \bigcap_{m>0} \left[(\m^m F_n) + L\right] \subseteq L.
$$
Consequently, $J_i \cap (a_{i+1}F_{n-n_{i+1}}) \subseteq L$. 
\medskip

In conclusion, for both cases, we have $J_i \cap (a_{i+1}F_{n-n_{i+1}}) \subseteq L$, which implies $J_i \cap F_n = L =\sum_{j=1}^{i}a_j F_{n-n_j}$. 
Next, it is straightforward to verify that the equality
$$
(a_1t^{n_1}, a_2 t^{n_2}, \ldots, a_i t^{n_i})G(\calF):_{G(\calF)} a_{i+1}t^{n_{i+1}} = (a_1t^{n_1}, a_2 t^{n_2}, \ldots, a_i t^{n_i})G(\calF)
$$
holds for all $0 \le i <r$. Consider a surjective graded ring homomorphism 
$$
\varphi : G(\calF) \to G(\calF/J_r) 
$$
defined by $\varphi(\overline{xt^n}) = \overline{(x+F_{n+1})t^n}$ for each $x \in F_n$. Since $\Ker \varphi = (a_1t^{n_1}, a_2 t^{n_2}, \ldots, a_r t^{n_r})G(\calF)$, we obtain the isomorphism
$G(\calF)/ (a_1t^{n_1}, a_2 t^{n_2}, \ldots, a_r t^{n_r})G(\calF) \cong G(\calF/J_r)$. Finally, because $a_1, a_2, \ldots, a_r \in \m$ and $F_1$ is $\m$-primary, we have $(F_1 + J_r)/J_r\ne R/J_r$. This implies $G(\calF/J_r) \ne (0)$. Therefore, $a_1t^{n_1}, a_2t^{n_2}, \ldots, a_rt^{n_r} \in \calR'(\calF)$ forms a regular sequence on $G(\calF)$.
\end{proof}



\begin{prop}\label{depth}
Let $(R, \m)$ be a Noetherian local ring and $\calF =\{F_n\}_{n \in \Bbb Z}$ the Hilbert filtration of ideals in $R$. Suppose that $\depth G(\calF) = r > 0$. Then there exists homogeneous elements $f_1, f_2, \ldots, f_r$ in $\calR'(\calF)$ of non-negative degree such that $f_1, f_2, \ldots, f_r$ forms a regular sequence on $G(\calF)$.
\end{prop}

\begin{proof}
Let $\fkN$ denote the graded maximal ideal of $G(\calF)$. We aim to prove that $\Ass G(\calF) \subseteq \Proj G(\calF)$. Indeed, note that $\fkN \notin \Ass G(\calF)$. For each $P \in \Ass G(\calF)$, the ideal $P$ is graded and satisfies $P \subsetneq \fkN$. Assume, for the sake of contradiction, that $P \supseteq G(\calF)_+$, where $G(\calF)_+ = \bigoplus_{n>0}F_n/F_{n+1}$. Since $F_1$ is $\m$-primary and $\m G(\calF)$ is contained in the integral closure of $G(\calF)_+$, it follows that $G(\calF)_+$ is a reduction of $\fkN$. Consequently, we would have $P = \fkN$, which is a contradiction. Hence, $P \in \Proj G(\calF)$. Therefore, we conclude that $\Ass G(\calF) \subseteq \Proj G(\calF)$, as desired.

By employing the same technique as in the proof of \cite[Lemma 3.1]{NO}, for any subset $\calX \subseteq \Ass G(\calF)$ and  any graded ideal $I$ of $G(\calF)$ with $I \not\subseteq \bigcup_{P \in \calX}P$, there exists a homogeneous element $g \in I$ such that $g \not\in \bigcup_{P \in \calX}P$. In particular, we can choose a homogeneous $g \in \fkN$ such that $g \not\in \bigcup_{P \in \Ass G(\calF)}P$. Write $g = \overline{a_1t^{n_1}}$, where $a_1 \in F_{n_1}$ and $n_1 \ge 0$, with $\overline{(-)}$ denoting the image in $G(\calF)$. By Proposition \ref{regseq}, it follows that $a_1 \in \m$ is $R$-regular and that 
$$
G(\calF)/(a_1t^{n_1})G(\calF) \cong G(\calF/(a_1)).
$$
Since the filtration $\calF/(a_1)$ is Hilbert and $\depth G(\calF/(a_1)) = r-1$, the induction hypothesis ensures the existence of homogeneous elements $g_2, \ldots, g_r \in G(\calF/(a_1))$ of non-negative degree such that these elements form a regular sequence on $G(\calF/(a_1))$. Consequently, there exist homogeneous elements $f_1, f_2, \ldots, f_r$ in $\calR'(\calF)$ of non-negative degree such that $f_1, f_2, \ldots, f_r$ is $G(\calF)$-regular.
\end{proof}


\section{When does the quasi-Gorenstein property deform?}

In this section, we explore the conditions under which the quasi-Gorenstein property deforms. Let $(R, \m)$ be a Noetherian local ring with $d=\dim R >0$ that admits a canonical module $\rmK_R$. Note that $R$ is Gorenstein if $d \le 3$ and $R/xR$ is quasi-Gorenstein for some non-zerodivisor $x \in \m$. Thus, when addressing the deformation problem of the quasi-Gorenstein property, it is sufficient to focus on cases where $d \ge 4$ and $R$ is not Cohen-Macaulay. In such cases, one has $\depth R \ge 3$.




We begin with the following, which plays a key in our argument. 

\begin{thm}\label{3.1}
Let $(R, \m)$ be a Noetherian local ring with $d = \dim R \ge 4$ and $\depth R=d-1$, admitting the canonical module $\rmK_R$. Suppose that  there exists a non-zerodivisor $x \in \m$ on $R$ such that $R/xR$ is quasi-Gorenstein and $\rmH^{d-2}_{\m}(R/xR)$ is finitely generated as an $R$-module. 
Then the following assertions hold true, where $M$ denotes the Matlis dual of $\rmH^{d-1}_{\m}(R)$.
\begin{enumerate}
\item[$(1)$] $R$ is quasi-Gorenstein if and only if $M$ is a Cohen-Macaulay $R$-module with $\dim_RM = 1$.
\item[$(2)$] 
$\Supp_R M$ is the non-Cohen-Macaulay locus of $R$, i.e., the set of prime ideals $\p$ of $R$ for which the local ring $R_\p$ is not Cohen-Macaulay.
\item[$(3)$] If $\dim_R M = 1$, then $M$ is a Cohen-Macaulay $R$-module if and only if the equality
$$
\ell_R(\rmH^{d-2}_{\m}(R/xR)) = \sum_{\p \in \Assh_R M}\ell_{R_{\p}}(\rmH^{d-2}_{\p R_{\p}}(R_{\p}))\cdot \rme_0(x, R/\p)
$$
holds.
\end{enumerate}
\end{thm}

\begin{proof}
Without loss of generality, we may assume that $R$ is $\m$-adically complete. Let $\overline{R}=R/xR$. Note that $M \ne (0)$. By applying the functor $\rmH^i_{\m}(-)$ to the sequence $
0 \to R \overset{x}{\to} R \to \overline{R} \to 0,
$
we obtain the  long exact sequence of $R$-modules
$$
0 \to \rmH^{d-2}_{\m}(\overline{R}) \to \rmH^{d-1}_{\m}(R) \overset{x}{\to} \rmH^{d-1}_{\m}(R) \to \rmH^{d-1}_{\m}(\overline{R}) \to \rmH^{d}_{\m}(R) \overset{x}{\to} \rmH^{d}_{\m}(R) \to 0.
$$
Taking Matlis duality yields the exact sequence
$$
0 \to \rmK_R \overset{x}{\to} \rmK_R \to \rmK_{\overline{R}} \to M 
 \overset{x}{\to} M \to [\rmH^{d-2}_{\m}(\overline{R})]^{\vee} \to 0
$$
which can be divided into the following three parts
$$
0 \to \rmK_R \overset{x}{\to} \rmK_R \to  \rmK_R/x \rmK_R \to 0, 
$$
$$
0 \to \rmK_R/x \rmK_R \to  \rmK_{\overline{R}} \to C \to 0, \ \ \text{and} 
$$
$$
0 \to C \to M  \overset{x}{\to} M \to [\rmH^{d-2}_{\m}(\overline{R})]^{\vee} \to 0
$$
where $C = (0):_M x$. Since  $[\rmH^{d-2}_{\m}(\overline{R})]^{\vee}$ has finite length, it follows that $\dim_RM \le 1$.

$(1)$ We first assume that $R$ is quasi-Gorenstein. If $\dim_R M = 0$, the ring $R$ has FLC, because $\ell_R(\rmH^{d-1}_\m(R))<\infty$. By \cite[(1.16)]{AG}, we get the isomorphisms 
$$
M = [\rmH^{d-1}_\m(R)]^{\vee} \cong \rmH^{d-(d-1)+1}_\m (\rmK_R) \cong \rmH^2_\m(R) = (0)
$$
which leads to a contradiction. Thus $\dim_RM = 1$, so  $x \in \m$ forms a system of parameters for $M$. Therefore, it suffices to prove that $C =(0)$. 
Assume, to the contrary, that $C \ne (0)$, and seek a contradiction. 
As $\ell_R(M/xM)<\infty$, for each $\p \in \Spec R \setminus\{\m\}$, we have the exact sequence
$$
0 \to C_\p \to M_\p \overset{x}{\to} M_\p \to 0
$$
of $R_\p$-modules. This implies that $C_\p = (0)$, and hence $\ell_R(C) < \infty$. Next, applying the depth lemma to the exact sequence
$$
0 \to \rmK_R/x\rmK_R \to \rmK_{\overline{R}}\cong \overline{R} \to C \to 0,
$$
we have $\depth \rmK_R = 2$. However, this contradicts the facts that $R \cong \rmK_R$ and $\depth R \ge 3$. 
Consequently $C =(0)$, as desired. 

Conversely, we assume $M$ is Cohen-Macaulay and of dimension one. Since $\ell_R(M/xM)<\infty$, note that $x \in \m$ is a non-zerodivisor on $M$. Thus $C=(0)$ and $\rmK_R/x\rmK_R \cong \overline{R}$. It is straightforward to check that the local ring $R$ is quasi-Gorenstein.

$(2)$ Let $\p \in \Supp_R M$. To show that $R_\p$ is not Cohen-Macaulay, we may assume $\p \ne \m$. Then $\dim_R M=1$ and $\p \in \Assh_R M$. In particular, $\Supp_RM \setminus\{\m\} \subseteq \Assh_RM$. Choose a Gorenstein complete local ring $S$ with $\dim S = d$ such that $R$ is a homomorphic image of $S$. Let $P = \p \cap S \in \Spec S$. Then $S/P \cong R/\p$, which implies  $\dim S_P = d-1$. As $S$ is Gorenstein and complete, we get the isomorphism
$$
M = [\rmH^{d-1}_\m(R)]^{\vee} \cong \Ext^1_S(R, S)
$$
which yields $M_\p \cong \Ext^1_{S_P}(R_{\p}, S_P)$. Hence
$$
(0) \ne \widehat{R_{\p}}\otimes_{R_\p}M_\p \cong \widehat{\,S_{P}}\otimes_{S_P}\Ext^1_{S_P}(R_{\p}, S_P) \cong [\rmH^{d-2}_{\p R_{\p}}(R_{\p})]^{\vee}.
$$ 
This shows $0 < \ell_{R_\p}(\rmH^{d-2}_{\p R_\p}(R_\p)) < \infty$, because $\p \in \Assh_RM$.  
Then $R_\p$ is not Cohen-Macaulay. Indeed, we assume the contrary, i.e., $R_\p$ is Cohen-Macaulay. As $\depth R_\p = d-2 \ge 2$, we can choose a non-zerodivisor $\alpha \in \p R_\p$ on $R_\p$. The sequence $0 \to R_\p \to R_\p \to R_\p/\alpha R_\p \to 0$ induces the exact sequence
$$
0 \to \rmH^{d-3}_{\p R_\p}(R_\p/\alpha R_\p) \to \rmH^{d-2}_{\p R_\p}(R_\p) \overset{\alpha}{\to} \rmH^{d-2}_{\p R_\p}(R_\p) \to  0
$$
of $R_\p$-modules. Since $\rmH^{d-2}_{\p R_\p}(R_\p)$ is finitely generated, it follows that $\rmH^{d-3}_{\p R_\p}(R_\p/\alpha R_\p)=(0)$. This contradicts the fact that $R_\p/\alpha R_\p$ is Cohen-Macaulay and of dimension $d-3$. Therefore, $R_\p$ cannot be Cohen-Macaulay, as required.

On the other hand, pick $\p \in \Spec R$ such that $R_\p$ is not Cohen-Macaulay. Let $S$ be  a Gorenstein complete local ring with $\dim S = d$ such that $R$ is a homomorphic image of $S$. Setting $P = \p \cap S$, $n=\dim S_P$, and $t=\depth R_\p$, we have
$$
n=\dim S_P \ge \dim R_\p > t=\depth R_\p,
$$
which implies $\Ext^{n-t}_S(R, S) \ne (0)$. If $n-t \ge 2$, then 
$$
\Ext^{n-t}_S(R, S) \cong [\rmH^{d-(n-t)}_\m(R)]^{\vee}=(0),
$$
a contradiction. Hence,  $n-t = 1$, and we conclude that 
$$
M_\p \cong \Ext^1_{S_P}(R_\p, S_P) \ne (0).
$$
Thus, $\Supp_RM$ is precisely the non-Cohen-Macaulay locus of $R$. 

$(3)$ Suppose $\dim_RM = 1$. Since $M/xM \cong [\rmH^{d-2}_\m(\overline{R})]^{\vee}$, we have $\ell_R(M/xM) = \ell_R(\rmH^{d-2}_\m(\overline{R}))$. First, assume that $M$ is Cohen-Macaulay. Because $x \in \m$ is a system of parameters for $M$, we get $\rme_0(x, M) = \ell_R(M/xM)$. Using the associativity formula for multiplicity, we derive the equalities
$$
\rme_0(x, M) = \sum_{\p \in \Assh_RM}\ell_{R_{\p}}(M_{\p}) \cdot \rme_0(x, R/\p) = \sum_{\p \in \Assh_RM}\ell_{R_{\p}}(\rmH^{d-2}_{\p R_{\p}}(R_\p)) \cdot \rme_0(x, R/\p)  
$$
where the second equality follows from the fact that
$$
\ell_{R_{\p}}(M_{\p}) = \ell_{R_\p}(\widehat{R_\p}\otimes_{R_\p}M_\p) = \ell_{R_{\p}}[(\rmH^{d-2}_{\p R_{\p}}(R_\p)]^{\vee})=\ell_{R_{\p}}(\rmH^{d-2}_{\p R_{\p}}(R_\p))
$$ 
for all $\p \in \Assh_RM$. This establishes the desired equality. Conversely, assume that
$$
\ell_R(\rmH^{d-2}_{\m}(R/xR)) = \sum_{\p \in \Assh_R M}\ell_{R_{\p}}(\rmH^{d-2}_{\p R_{\p}}(R_{\p}))\cdot \rme_0(x, R/\p).
$$
Since $\ell_{R_{\p}}(M_{\p}) = \ell_{R_{\p}}(\rmH^{d-2}_{\p R_{\p}}(R_{\p}))$, it follows that $\ell_R(M/xM) = \rme_0(x, M)$. Hence, $M$ is Cohen-Macaulay. 
\end{proof}

\begin{rem}
Under the assumption of Theorem \ref{3.1}, the condition that $\rmH^{d-2}_{\m}(R/xR)$ is finitely generated as an $R$-module is equivalent to stating that the ring $R/xR$ has FLC. The latter condition holds if either $R/xR$ is locally Cohen-Macaulay on the punctured spectrum, or $\dim R=4$.
\end{rem}


\begin{cor}
Let $(R, \m)$ be a Noetherian local ring with $\dim R = 4$ admitting the canonical module $\rmK_R$. 
Suppose that $R$ is not Cohen-Macaulay and there exists a non-zerodivisor $x \in \m$ on $R$ such that $R/xR$ is quasi-Gorenstein. Then the following assertions hold true, where $M$ denotes the Matlis dual of $\rmH^{3}_{\m}(R)$. 
\begin{enumerate}
\item[$(1)$] $R$ is quasi-Gorenstein if and only if $M$ is a Cohen-Macaulay $R$-module with $\dim_RM = 1$.
\item[$(2)$] 
$\Supp_R M$ is the non-Cohen-Macaulay locus of $R$.
\item[$(3)$] If $\dim_R M = 1$, then $M$ is a Cohen-Macaulay $R$-module if and only if the equality
$$
\ell_R(\rmH^{2}_{\m}(R/xR)) = \sum_{\p \in \Assh_R M}\ell_{R_{\p}}(\rmH^{2}_{\p R_{\p}}(R_{\p}))\cdot \rme_0(x, R/\p)
$$
holds.
\end{enumerate}
\end{cor}

\begin{proof}
We may assume that $R$ is complete. Set $\overline{R} = R/xR$. Note that $\depth R = 3$. For every $P \in \Spec \overline{R} \setminus \{\overline{\m}\}$, we have $\dim \overline{R}_P \le 2$; hence $\overline{R}_P$ is Cohen-Macaulay because $\overline{R}$ satisfies Serre's $(S_2)$ condition.  
Since $(\rmK_{\overline{R}})_P \cong \overline{R}_P \ne (0)$, it follows that $P \in \Supp_{\overline{R}}\rmK_{\overline{R}}$. Thus, $3 = \dim \overline{R}_P + \dim \overline{R}/P$. As $\overline{R}$ is a homomorphic image of a Cohen-Macaulay ring, $\overline{R}$ has FLC. Therefore, $\rmH^2_{\m}(\overline{R})$ is finitely generated as an $R$-module. Hence, the assertions follow from Theorem \ref{3.1}.
\end{proof}

\begin{cor}
Let $(R, \m)$ be a Noetherian local ring with $d = \dim R >0$ admitting the canonical module $\rmK_R$. Suppose that $R$ has FLC and there exists a non-zerodivisor $x \in \m$ on $R$ such that $R/xR$ is quasi-Gorenstein. Then $R$ is quasi-Gorenstein if and only if $\rmH^{d-1}_\m(R) = (0)$. In particular, $R$ is quasi-Gorenstein if and only if $R$ is Gorenstein, provided $d \le 4$. 
\end{cor}

\begin{proof}
Let $\overline{R} = R/xR$. For $d \geq 3$, the ring $R$ is Gorenstein, so we may assume $d \geq 4$. Hence, $\depth R \ge 3$. Suppose that $\rmH^{d-1}_\m(R) = (0)$. Applying the functor $\rmH^{i}_\m(-)$ to the exact sequence $0 \to \rmK_R \overset{x}{\to} \rmK_R \to \overline{R} \to 0$, we obtain the sequence $0 \to \rmH^{d-1}_\m(\overline{R}) \to \rmH^{d}_\m(R)  \overset{x}{\to} \rmH^{d}_\m(R)  \to 0$. This implies the exact sequence
$$
0 \to \rmK_R \overset{x}{\to} \rmK_R \to \overline{R} \to 0 
$$
of $R$-modules. Therefore, we have $R \cong \rmK_R$. Conversely, assume $R$ is quasi-Gorenstein. Then we have the exact sequence
$$
0 \to \overline{R} \to \overline{R} \to C \to 0
$$ 
where $M$ denotes the Matlis dual of $\rmH^{d-1}_\m(R)$ and $C=(0):_M x$. Suppose, for the sake of contradiction, that $\rmH^{d-1}_\m(R) \ne (0)$. Then $M \ne (0)$, and $\ell_R(M) = \ell_R(\rmH^{d-1}_\m(R))$ is finite because $R$ has FLC. Thus, $\dim_RM=0$. Consequently, $C \ne (0)$, which implies $\dim_RC = 0$ and $\depth R = 2$. This is a contradiction. Hence, $\rmH^{d-1}_\m(R) = 0$, as required.
Finally, assume $d = 4$ and $R$ is quasi-Gorenstein. Since $\rmH^{3}_\m(R) = 0$ and $\depth R \geq 3$, it follows that $\depth R \geq 4$. Hence, $R$ is Cohen-Macaulay, which means $R$ is Gorenstein. 
\end{proof}

To conclude this section, we articulate the criteria under which a quasi-Gorenstein ring possessing the FLC property is guaranteed to be Gorenstein.

\begin{rem}
Let $(R, \m)$ be a Noetherian local ring with $d = \dim R \ge 2$ that admits the canonical module $\rmK_R$. Suppose $R$ is quasi-Gorenstein having FLC. Then $R$ is Gorenstein if and only if $\depth R \ge \frac{d}{2} + 1$. 
\end{rem}

\begin{proof}
We may first assume that $R$ is complete. Suppose $R$ is Gorenstein. Since $d \ge 2$, we have $d \ge \frac{d}{2} + 1$; hence, $\depth R \ge \frac{d}{2} + 1$. Conversely, when $d = 2$, the ring $R$ is Cohen-Macaulay because $\depth R \ge \frac{d}{2} + 1 = 2$. This implies that $R$ is Gorenstein. Therefore, we may assume $d \ge 3$. Since $R$ is complete and has FLC,  we have the isomorphisms
$$
\rmH^i_\m(R) \cong [\rmH^{d-i+1}_\m(\rmK_R)]^{\vee} \cong [\rmH^{d-i+1}_\m(R)]^{\vee}
$$
for every $2 \le i \le d-1$. Set $t =\depth R$. Then $t \ge \frac{d}{2} + 1$, which implies that $t \ge 3$. Assume $R$ is not Cohen-Macaulay. Then $t \le d-1$. This leads us to $(0) \ne \rmH^t_{\m}(R) \cong [\rmH^{d-t+1}_\m(R)]^{\vee}$. Thus, $\rmH^{d-t+1}_\m(R) \ne (0)$. Consequently, $t \le d-t+1$. This yields $\frac{d}{2} + 1 \le t \le \frac{d+1}{2}$, which is a contradiction. Therefore, $R$ must be Cohen-Macaulay, and hence $R$ is Gorenstein.
\end{proof}


\section{Proof of main theorem}

First, we establish the notation and assumptions that form the basis of all the results in this section.
\begin{setup}
Let $(R, \m)$ be a Noetherian local ring, and let $\calF=\{F_n\}_{n \in \Bbb Z}$ denote a Noetherian filtration of ideals of $R$ with $F_1 \ne R$. Let $t$ be an indeterminate over $R$. 
The {\it extended Rees algebra of $\calF$} and the {\it associated graded ring of $\calF$} are defined as
$$
\calR'(\calF) = \sum_{n \in \Bbb Z}F_nt^n \subseteq R[t, t^{-1}], \ \ G(\calF) = \calR'(\calF)/t^{-1}\calR'(\calF) \cong \bigoplus_{n \ge 0}F_n/F_{n+1},
$$
respectively. Let $\fkM$ and $\fkN$ denote the unique graded maximal ideals of $\calR'(\calF)$ and $G(\calF)$, respectively. For each $\p \in \Spec R$, we define
$
\p' = \p \cdot R[t, t^{-1}]\cap \calR'(\calF).
$
\end{setup}

Note that $\p'$ is a graded prime ideal of $\calR'(\calF)$ with $t^{-1} \not\in \p'$ and $\p' \subsetneq \fkM$. There exists a one-to-one correspondence between $\Spec R$ and ${}^*D(t^{-1})$, the set of all graded prime ideals $P$ of $\calR'(\calF)$ such that $t^{-1} \not\in P$, where $\p \in \Spec R$ corresponds to $\p' \in {}^*D(t^{-1})$. Furthermore, we have the isomorphisms
$$
\calR'(\calF)_{\p'} \cong R[t, t^{-1}]_{\p R[t, t^{-1}]} \cong R[t]_{\p R[t]} \ \ \ \text{and} \ \ \ \calR'(\calF)/\m' \cong (R/\m)[t^{-1}].
$$


\begin{lem}\label{4.2}
Suppose the following conditions hold.
\begin{enumerate}
\item[$(1)$] $R$ is not a Cohen-Macaulay ring but possesses FLC.
\item[$(2)$] $F_1$ is an $\m$-primary ideal of $R$.
\item[$(3)$] $G(\calF)_Q$ is Cohen-Macaulay for every graded prime ideal $Q$ of $G(\calF)$ with $Q \ne \fkN$.
\end{enumerate}
Then, for every $P \in \Spec \calR'(\calF)$ such that $P$ is graded and $P \ne \fkM$, the local ring $\calR'(\calF)_P$ is not Cohen-Macaulay if and only if $P = \m'$. In particular, $V(\m')$ coincides with the non-Cohen-Macaulay locus of $\calR'(\calF)$.
\end{lem}

\begin{proof}
Let $P$ be a graded prime ideal of $\calR'(\calF)$ with $P \ne \fkM$. Suppose that $\calR'(\calF)_P$ is not Cohen-Macaulay. If $t^{-1} \in P$, then $P$ contains $F_n = F_nt^n \cdot t^{-n}$ for all $n > 0$, which implies $F_{\ell}t^{\ell} \nsubseteq P$ for some $\ell > 0$. By our assumption, $G(I)_Q$ is Cohen-Macaulay, where $Q = P \cdot G(I) \subsetneq \fkN$. Hence, $\calR'(I)_P$ must be a Cohen-Macaulay local ring, which leads to a contradiction. Therefore, $t^{-1} \notin P$. Thus, $P \in {}^*D(t^{-1})$, and there exists $\p \in \Spec R$ such that $P = \p'$.
This shows that
$$
\calR'(\calF)_P \cong \calR'(\calF)_{\p'} \cong R[t, t^{-1}]_{\p R[t, t^{-1}]} \cong R[t]_{\p R[t]}
$$
is not Cohen-Macaulay. Consequently,  $R_\p$ is not Cohen-Macaulay either. Since $R$ has FLC, it follows that $\p =\m$. Hence, we conclude that  $P = \m'$. Conversely, if we assume $P=\m'$, then the isomorphism
$
\calR'(\calF)_P \cong R[t]_{\m R[t]}
$
guarantees that $\calR'(\calF)_P$ is not Cohen-Macaulay. 
Let us verify the last assertion. For each $P \in V(\m')$, there exists an isomorphism $(\calR'(\calF)_P)_{\m'\calR'(\calF)_P} \cong \calR'(\calF)_\m'$. In particular, $\calR'(\calF)_P$ is not Cohen-Macaulay. Conversely, suppose $P \in \Spec \calR'(\calF)$ such that $\calR'(\calF)_P$ is not Cohen-Macaulay. Let $P^*$ denote the ideal of $\calR'(\calF)$ generated by the homogeneous elements of $P$. 
Observe that $P^* \subseteq \fkM$. If $P^*=\fkM$, then $P=\fkM$. Suppose $P^* \ne \fkM$. Then $P^* = \m'$, because $\calR'(\calF)_{P^*}$ is not Cohen-Macaulay. In either case, we have $P \subseteq \m'$, and thus $P \in V(\m')$.
\end{proof}

We now present the main result of this paper. 

\begin{thm}
Let $(R, \m)$ be a Noetherian local ring with $d=\dim R \ge 3$ which is a homomorphic image of a Gorenstein ring. Let  $\calF=\{F_n\}_{n \in \Bbb Z}$ be the Hilbert filtration of ideals in $R$. Suppose that $G(\calF)$ is a quasi-Gorenstein graded ring, $\depth \calR'(\calF) \ge d$, and $\rmH^{d-1}_{\fkM}(G(\calF))$ is finitely generated as an $\calR'(\calF)$-module. Then the following conditions are equivalent.
\begin{enumerate}
\item[$(1)$] $\calR'(\calF)$ is a quasi-Gorenstein graded ring.
\item[$(2)$] The equality $\ell_R(\rmH^{d-1}_{\m}(R)) = \ell_{G(\calF)}(\rmH^{d-1}_{\fkM}(G(\calF)))$ holds. 
\end{enumerate}
\end{thm}

\begin{proof}
Set $\calR' = \calR'(\calF)$ and $G=G(\calF)$. The ring $\calR'$ admits the graded canonical module $\rmK_{\calR'}$. 

$(2) \Rightarrow (1)$ We may assume that $R$ is not Cohen-Macaulay. Indeed, if $R$ were Cohen-Macaulay, then $G$ would also be Cohen-Macaulay due to the equality $\ell_R(\rmH^{d-1}_{\m}(R)) = \ell_{G}(\rmH^{d-1}_{\fkM}(G))$ and the fact that $\depth G \ge d-1$. Consequently, $G$ would be Gorenstein, which in turn implies that $\calR'$ is Gorenstein as well. 
Now, assume for contradiction that $\calR'$ is Cohen-Macaulay. The Cohen-Macaulayness of $G$ would imply that the localization $R_\p$ is Cohen-Macaulay for every $\p \in V(F_1)$. However, this leads to a contradiction, as it conflicts with the facts that $F_1$ is $\m$-primary and $R$ is not Cohen-Macaulay. Therefore, $\calR'$ cannot be Cohen-Macaulay. We conclude, in particular, that $\depth \calR' = d$. 

We set $A = \calR'_{\fkM}$. Then $\dim A=d+1 \ge 4$ and $\depth A = d$. Note that $G_{\fkN} \cong G_{\fkM} \cong A/t^{-1}A$. 
Let $M$ denote the Matlis dual of $\rmH_{\fkM A}^{d}(A)$. Observe that $M \ne (0)$. Applying the functor $\rmH_{\fkM A}^i(-)$ to the exact sequence
$0 \to A \overset{t^{-1}}{\to} A \to G_{\fkM} \to 0$
of $A$-modules, we obtain the long exact sequence
$$
0 \to \rmH^{d-1}_{\fkM A}(G_\fkM) \to \rmH^d_{\fkM A}(A) \overset{t^{-1}}{\to} \rmH^d_{\fkM A}(A) \to \rmH^{d}_{\fkM A}(G_\fkM)  \to \rmH^{d+1}_{\fkM A}(A)\overset{t^{-1}}{\to}  \rmH^{d+1}_{\fkM A}(A) \to 0.
$$
Taking the Matlis dual of this sequence yields
$$
0 \to \rmK_{A} \overset{t^{-1}}{\to} \rmK_{A} \to \rmK_{G_{\fkM}} \to M \overset{t^{-1}}{\to} M \to [\rmH^{d-1}_{\fkM A}(G_{\fkM})]^{\vee} \to 0.
$$
Since $\ell_A(M/t^{-1}M) = \ell_A(\rmH^{d-1}_{\fkM A}(G_\fkM))$ is finite, we conclude that $\dim_A M \le 1$. If $\dim_A M=0$, the ring $A$ has FLC. Consequently, $A$ is locally Cohen-Macaulay on the punctured spectrum, and thus 
$$
A_{\m'A} \cong (\calR')_{\m'} \cong R[t, t^{-1}]_{\m R[t, t^{-1}]} \cong R[t]_{\m R[t]}
$$
is Cohen-Macaulay. By flat descent, this would imply that $R$ is Cohen-Macaulay, which contradicts our assumption. Therefore, we must have $\dim_AM=1$. 

Observe that $\depth G = d-1>0$. By Proposition \ref{depth}, it follows that $\depth R = d-1$. Since both $\rmH^{d-1}_\m(R)$ and $\rmH^{d-1}_\fkM(G)$ are finitely generated, the rings $R$ and $G_{\fkN}$ have FLC. We now claim that $\Assh_{A}M=\{\m' A\}$. 
To establish this, let $P \in \Assh_{A}M$. By Theorem \ref{3.1}, the local ring $A_P$ is not Cohen-Macaulay. Set $\p = P \cap \calR'$. Then $\p \in \Spec \calR'$ and $\p \subseteq \fkM$. Since $P \in \Assh_AM$ and $\dim_AM = 1$, we deduce that $\p \ne \fkM$. Note further that $A_P \cong \calR'_{\p}$. Let $\p^*$ denote the ideal of $\calR'$ generated by the homogeneous elements of $\p$. It follows that $\calR'_{\p^*}$ is not Cohen-Macaulay. By Lemma \ref{4.2}, we conclude that $\p^* = \m'$. Hence $\m' \subseteq \p \subsetneq \fkM$. Since $\dim \calR'/\m' = 1$, it must be that $\m' = \p$, and thus $P = \m' A$. Consequently, $\Assh_{A}M=\{\m' A\}$, as claimed.

By setting $R(t) = R[t]_{\m R[t]}$, we obtain
\begin{eqnarray*}
\sum_{P \in \Assh_AM}\ell_{A_P}(\rmH_{PA_P}^{d-1}(A_P))\cdot \rme_0(t^{-1}, A/P) 
\!\!\!\!&=&\!\!\!\! \ell_{A_{\m' A}}(\rmH_{{(\m' A)}A_{\m' A}}^{d-1}(A_{\m' A}))\cdot \rme_0(t^{-1}, A/{\m' A}) \\[-7pt]
\!\!\!\!&=&\!\!\!\!  \ell_{R(t)}(\rmH_{\m R(t)}^{d-1}(R(t))) =  \ell_{R}(\rmH_{\m}^{d-1}(R))\\ \!\!\!\!&=&\!\!\!\! \ell_{G}(\rmH^{d-1}_{\fkM}(G)) = \ell_A(\rmH_{\fkM A}^{d-1}(A/t^{-1}A)). 
\end{eqnarray*}
Here, the second equality follows from the fact that 
$$
A/\m' A \cong (\calR'/\m')_{\fkM} \cong ((R/\m)[t^{-1}])_{(t^{-1})},
$$
which is a regular local ring. By Theorem \ref{3.1}, we conclude that $M$ is Cohen-Macaulay. Consequently, the local ring $A=\calR'_{\fkM}$ is quasi-Gorenstein, and therefore $\calR'$ is also quasi-Gorenstein.

$(1) \Rightarrow (2)$ 
Suppose $\calR'$ is Cohen-Macaulay. Then both $G$ and $R$ are Cohen-Macaulay, and hence $\ell_R(\rmH^{d-1}_{\m}(R)) = 0 = \ell_{G}(\rmH^{d-1}_{\fkM}(G))$. Therefore, we may assume that $\calR'$ is not Cohen-Macaulay, in which case $\depth \calR' = d$.
Let $A = \calR'_{\fkM}$, and denote by $M$ the Matlis dual of $\rmH_{\fkM A}^{d}(A)$. Then $M \ne (0)$. Since $A$ is quasi-Gorenstein, Theorem \ref{3.1} ensures that $M$ is Cohen-Macaulay of dimension one. Consequently, the equality
$$
\ell_{G}(\rmH^{d-1}_{\fkM}(G))=\ell_A(\rmH_{\fkM A}^{d-1}(A/t^{-1}A)) = \sum_{P \in \Assh_AM} \ell_{A_P}(\rmH_{PA_P}^{d-1}(A_P))\cdot \rme_0(t^{-1}, A/P)
$$
holds. 
We claim that $\Assh_{A}M =\{\m' A\}$. To verify this, let $P \in \Assh_{A}M$. By Theorem \ref{3.1}, the localization $A_P$ is not Cohen-Macaulay. Let $\p = P \cap \calR'$. Then $\p \subseteq \fkM$. If $\p = \fkM$, then $P = \fkm A$, which contradicts the assumption that $P \in \Assh_AM$ and $\dim_AM=1$. Thus $\p \subsetneq \fkM$. Next, let $\p^*$ denote the ideal of $\calR'$ generated by the homogeneous elements of $\p$. It follows that $\calR'_{\p^*}$ is not Cohen-Macaulay, and consequently $A_{\p^* A}$ is not Cohen-Macaulay either. By Theorem \ref{3.1} again, we have $\p^*A \in \Supp_AM$. 
Since $M$ is Cohen-Macaulay, it holds that $\dim_{A_{\p^*A}}M_{\p^*A} \le 1$. 
If $\dim_{A_{\p^*A}}M_{\p^*A}=1$, then $\dim A/\p^* A=0$, which is a contradiction because $\p^*A \subseteq \p A \subsetneq \fkM A$. Thus, $\dim_{A_{\p^*A}}M_{\p^*A}=0$, and so $\dim A/\p^* A=1$. 
Consider the exact sequence
$$
0 \to \rmK_{A} \overset{t^{-1}}{\to} \rmK_{A} \to \rmK_{G_{\fkM}} \to M \overset{t^{-1}}{\to} M \to [\rmH^{d-1}_{\fkM A}(G_{\fkM})]^{\vee} \to 0.
$$
This implies that $\ell_A(M/t^{-1}M) < \infty$, so $t^{-1}$ acts as a non-zerodivisor on $M$. In particular,  $t^{-1} \not\in \p^*$. Thus $\p^* \in {}^*D(t^{-1})$. There exists $\fkq \in \Spec R$ such that $\fkq' = \p^*$. Therefore
$$
\p^*A = \fkq'A \subseteq \m' A \subsetneq \fkM A,
$$
and since $\dim A/\p^* A=1$, we deduce $\p^*A = \q' A = \m' A$. As $\m'A=\p^*A \subseteq \p A=P\subsetneq \fkM A$, it follows that $\m'A=\p A = P$, because $\dim A/\m' A = 1$. Consequently, $\Assh_{A}M =\{\m' A\}$. Finally, we obtain
$$
\ell_{G}(\rmH^{d-1}_{\fkM}(G)) = \sum_{P \in \Assh_AM}\ell_{A_P}(\rmH_{PA_P}^{d-1}(A_P))\cdot \rme_0(t^{-1}, A/P) 
 = \ell_{R}(\rmH_{\m}^{d-1}(R)).
$$
as desired. This completes the proof.
\end{proof}

A direct application of Theorem \ref{main} yields the following result.

\begin{cor}\label{4.4}
Let $(R, \m)$ be a Noetherian local ring with $\dim R = 3$ and $\calF=\{F_n\}_{n \in \Bbb Z}$ the Hilbert filtration of ideals in $R$.  Suppose that $R$ is a homomorphic image of a Gorenstein ring and $G(\calF)$ is quasi-Gorenstein. Then the following conditions are equivalent.
\begin{enumerate}
\item[$(1)$] $\calR'(\calF)$ is a quasi-Gorenstein graded ring. 
\item[$(2)$] The equality $\ell_R(\rmH^2_{\m}(R)) = \ell_{G(\calF)}(\rmH^2_{\fkM}(G(\calF)))$ holds.
\end{enumerate}
\end{cor}

Recall that a Noetherian graded $k$-algebra $R = \bigoplus_{n \ge 0} R_n$ over a field $R_0=k$ is called {\it homogeneous}, if $R =k[R_1]$, i.e., $R$ is generated by $R_1$ as a $k$-algebra.

\begin{cor}
Let $R=k[R_1]$ be a quasi-Gorenstein homogeneous ring over a field $k$ with $\dim R=3$. Then the extended Rees algebra $\calR'(\m R_\m)$ is quasi-Gorenstein, where $\m=R_+$ denotes the graded maximal ideal of $\calR'(\m)$. 
\end{cor}
\begin{proof}
As $R$ is homogeneous, we have $R \cong G(\m)$. Thus, $G(\m R_\m) \cong R_\m \otimes_R G(\m)\cong G(\m)_\m$ is quasi-Gorenstein. The $\m R_\m$-adic filtration is Hilbert, and the local ring $R_\m$ is a homomorphic image of a Gorenstein ring. By Corollary \ref{4.4}, the ring $\calR'(\m R_\m)$ is quasi-Gorenstein. 
\end{proof}

\begin{ex}
Let $\Delta$ be a two-dimensional finite abstract simplicial complex whose geometric realization is homeomorphic to an orientable manifold that is not a sphere. Then the Stanley-Reisner ring $R=k[\Delta]$, defined over a field $k$ of characteristic $0$, is quasi-Gorenstein but not Gorenstein (\cite[Remark 4.6]{VY}). 
Denote by $\m=R_+$ the graded maximal ideal of $R$. 
Consequently,   $\calR'(\m R_\m)$ is a non-Gorenstein quasi-Gorenstein ring.
\end{ex}


\begin{ac}
The author would like to thank K.-i. Watanabe for his valuable comments.  
\end{ac}







\begin{thebibliography}{20}

\bibitem{A}
{\sc Y. Aoyama}, Some basic results on canonical modules, {\em J. Math. Kyoto Univ.}, {\bf 23} (1983), no. 1, 85--94.

\bibitem{AG}
{\sc Y. Aoyama and S. Goto}, Some special cases of a conjecture of Sharp, {\em J. Ma th. Kyoto Univ.}, {\bf 26} (1986), no.4, 613--634.  

\bibitem{GO}
{\sc S. Goto and T. Ogawa}, A note on rings with finite local cohomology, {\em Tokyo J. Math.}, {\bf 6} (1983), no.2, 403--411.



\bibitem{HKU}
{\sc W. Heinzer, M.-K. Kim, and B. Ulrich}, The Gorenstein and complete intersection properties of associated graded rings, {\em J. Pure Appl. Algebra}, {\bf  201} (2005), 264--283.

\bibitem{HKU2}
{\sc W. Heinzer, M.-K. Kim, and B. Ulrich}, The Cohen-Macaulay and Gorenstein properties of rings associated to filtrations, {\em Comm. Algebra}, {\bf  39} (2011), no.10, 3547--3580.

\bibitem{HIO}
{\sc M. Herrmann, S. Ikeda, and U. Orbanz}, Equimultiplicity and Blowing Up, {\em Springer-Verlag, Berlin}, 1988.
 

\bibitem{HK}
{\sc J. Herzog and E. Kunz}, Der kanonische Modul eines Cohen-Macaulay-Rings, Lecture Notes in Mathematics, 238, {\em Springer-Verlag, Berlin-New York}, 1971.

\bibitem{H}
{\sc H. Hironaka}, Certain numerical characters of singularities, {\em J. Math. Kyoto Univ.}, {\bf 10} (1970), 151--187.

\bibitem{HM}
{\sc S. Huckaba and T. Marley}, Hilbert coefficients and the depths of associated graded rings, {\em J. London Math. Soc.}, {\bf 56} (1997), 64--76.

\bibitem{K}
{\sc T. Kawasaki}, On arithmetic Macaulayfication of Noetherian rings, {\em Trans. Amer. Math. Soc.}, {\bf 354} (2002), no.1, 123--149.



\bibitem{Kim}
{\sc Y. Kim}, Quasi-Gorensteinness of extended Rees algebras, {\em Comm. Algebra}, {\bf  9} (2017), no.4, 3547--3580.


\bibitem{NO}
{\sc C. Nastasescu and F. van Oystaeyen}, Graded and filtered rings and modules, Lecture Notes in Mathematics, 758, {\em Springer-Verlag, Berlin-New York}, 1979.

\bibitem{PS}
{\sc E. Platte and U. Storch}, Invariante regul\"{a}re Differentialformen auf Gorenstein-Algebren, {\em Math. Z.}, {\bf 157} (1977), 1--11.

\bibitem{Rees}
{\sc D. Rees}, A note on analytically unramified local rings, {\em J. London Math. Soc.}, {\bf 36} (1961), 24--28.

\bibitem{Re}
{\sc I. Reiten}, The converse to a theorem of Sharp on Gorenstein modules, {\em Proc. Amer. Math. Soc.}, {\bf 32} (1972), 417--420.

\bibitem{STC}
{\sc P. Schenzel, N. V. Trung, and N. T. Cuong}, Verallgeminerte Cohen-Macaulay-
Moduln, {\em Math. Nachr.}, {\bf 85} (1978), 57--73. 

\bibitem{S2}
{\sc R. Y. Sharp}, On Gorenstein modules over a complete Cohen-Macaulay ring, {\em Quart. J. Math.}, {\bf 22}, no. 3 (1971), 425--434.

\bibitem{STT}
{\sc K. Shimomoto, N. Taniguchi, and E. Tavanfar}, A study of quasi-Gorenstein rings II: Deformation of quasi-Gorenstein property, {\em J. Algebra}, {\bf 562} (2020), 368--389.

\bibitem{SV}
{\sc J. St\"{u}ckrad and W. Vogel}, Buchsbaum rings and applications, An interaction
between algebra, geometry and topology, {\em Springer-Verlag, Berlin}, 1986.

\bibitem{T}
{\sc N. V. Trung}, Toward a theory of generalized Cohen-Macaulay modules, {\em Nagoya Math. J.}, {\bf 102} (1986), 1--49.

\bibitem{VV}
{\sc P. Valabrega and G. Valla}, Form rings and regular sequences, {\em Nagoya Math. J.}, {\bf 72} (1978), 93--101.

\bibitem{VY}
{\sc M. Varbaro and H. Yu}, Lefschetz duality for local cohomology, {\em J. Algebra}, {\bf 639} (2024), 498--515.

\end{thebibliography}
\end{document}